\documentclass[10pt,leqno]{amsart}
\usepackage{amsmath}
\usepackage{amsfonts}
\usepackage{amssymb}
\usepackage{graphicx}
\usepackage{color}
\usepackage{hyperref}
\usepackage{xcolor}

\newtheorem{theorem}{Theorem}[section]
\newtheorem*{theorem*}{Theorem}
\newtheorem{lemma}{Lemma}[section]

\newtheorem{conjecture}[theorem]{Conjecture}

\def\p{\partial}

\def\R{\mathbb{R}}

\def\vp{\varphi}

\numberwithin{equation}{section}

\begin{document}

\title[Neumann eigenvalues in ROSS]{Isoperimetric inequalities for Neumann eigenvalues on bounded domains in rank-1 symmetric spaces }

\author{Yifeng Meng} \thanks{}
\address{School of Mathematical Sciences,
Fudan University, Shanghai, 200433, China}
\email{yfmeng23@m.fudan.edu.cn}

\author{Kui Wang} \thanks{The research of the second author is supported by Natural Science Foundation of Jiangsu Province Grant No. BK20231309}
\address{School of Mathematical Sciences, Soochow University, Suzhou, 215006, China}
\email{kuiwang@suda.edu.cn}



\begin{abstract}
    In this paper, we prove sharp  isoperimetric inequalities for lower order eigenvalues of Neumann Laplacian on  bounded domains in both compact and noncompact rank-1 symmetric spaces.  Our results generalize the work of Wang and Xia  for bounded domains in the hyperbolic space \cite{WX23}, and  Szeg\"o-Weinberger inequalities in rank-1 symmetric spaces obtained by Aithal and Santhanam  \cite{AS96}. 
\mbox{}
\end{abstract}
\maketitle

\section{Introduction}
    Let $\Omega$ be a bounded domain with smooth boundary, and  we consider the following  Neumann eigenvalue problem
    \begin{align}\label{1.1}
        \begin{cases}
            -\Delta u(x)=\mu(\Omega) u(x), &\quad x\in \Omega,\\
            \p_\nu u(x)=0, &\quad x\in \p \Omega,
        \end{cases}
    \end{align}
    where  $\nu$ is the outward unit normal to $\p \Omega$. It is well known that  the spectrum of  problem \eqref{1.1} is discrete and denoted by 
    \begin{align*}
       0= \mu_0(\Omega) < \mu_1(\Omega) \leq \mu_2(\Omega) \leq \cdots \to +\infty
    \end{align*}
with its multiplicity. For a bounded domain in $\R^n$, the classic Szeg\"o-Weinberger inequality asserts that the ball is  the unique maximizer of the first nonzero Neumann eigenvalue among domains having  the same volume in $\R^n$, namely,
 \begin{align}\label{1.2}
        \mu_1(\Omega)  \leq \mu_1(B),
\end{align}
where $B$ is a round ball in $\R^n$ with volume equal to $|\Omega|$ (the volume of $\Omega$). This result was proved by Szeg\"o \cite{Sz54} in planer domains, and by Weinberger \cite{Wei56} for all dimensions. Moreover inequality \eqref{1.2} also holds for bounded domains in hyperbolic space, a hemisphere \cite{AB95, Xu95} and $94\%$ of the $2$-sphere \cite{LL23}. Recently, the second named author \cite{Wang19} proved  a  Szeg\"o-Weinberger type inequality for bounded domains in curved spaces.

For bounded simply connected planer domains, applying Szeg\"o's conformal map technique one can prove
\begin{align}\label{1.3}
        \frac{1}{\mu_1(\Omega)}+\frac{1}{\mu_2(\Omega)}\geq \frac{2}{\mu_1(B)},
\end{align}
where $B$ is a round disk with $|B|=|\Omega|$.
Moreover estimate \eqref{1.3} is sharp and the equality holds on round balls. In 1993,  Ashbaugh and Benguria \cite{AB93} improved \eqref{1.3} by removing the assumption of simply connectedness, and conjectured that
\begin{conjecture}\label{con1}
For any bounded domain $\Omega$ in $\R^n$ ($n\ge 3$), it holds  
   \begin{align*}
            \frac{1}{\mu_1(\Omega) }+\frac{1}{\mu_2(\Omega) }+\cdots + \frac{1}{\mu_{n}(\Omega) }\geq \frac{n}{\mu_1(B)},
        \end{align*}
        where $B$ is a round ball in $\R^n$ with  $|B|=|\Omega|$. 
Moreover the equality occurs if and only if $\Omega$ is a round ball.  
\end{conjecture}
We refer to \cite{As99, He06} and references therein for some progress towards this conjecture. Recently, Wang and Xia  \cite{WX23} made an important progress, which supports Conjecture \ref{con1} well. Precisely, they proved  
\begin{theorem}[\cite{WX23}]\label{thm}
    For any bounded domain $\Omega$ in $\R^n$ ($n\ge 3$), it holds  
   \begin{align}\label{1.4}
            \frac{1}{\mu_1(\Omega) }+\frac{1}{\mu_2(\Omega) }+\cdots + \frac{1}{\mu_{n-1}(\Omega) }\geq \frac{n-1}{\mu_1(B)},
        \end{align}
        where $B$ is a round ball in $\R^n$ having the same volume as $\Omega$. 
Moreover the equality occurs if and only if $\Omega$ is a round ball.  
\end{theorem}
 Wang and Xia \cite{WX23} also proved  that inequality \eqref{1.4} holds true for bounded domains in hyperbolic space; Benguria, Brandolini and Chiacchio  proved \eqref{1.4} for bounded domains in a hemisphere \cite{BBC20} later. Besides, Wang and Xia \cite{WX21} proved a similar result as \eqref{1.4} for the ratio of Dirichlet eigenvalues.  

In 1996 Aithal and Santhanam \cite{AS96} proved that Szeg\"o-Weineberger inequality holds true for bounded domains in noncompact rank-1 symmetric spaces, and  domains contained in a geodesic ball of radius $i(M)/4$ in compact rank-1 symmetric spaces. Here and thereafter $i(M)$ denotes the injectivity radius of $M$ when $M$ is compact. Naturally, one may consider the Conjecture \ref{con1} in rank-1 symmetric spaces.
 Regarding Theorem \ref{thm},  we consider Wang-Xia type estimate \eqref{1.4} on bounded domains in the remaining rank-1 symmetric spaces, and 
prove sharp isoperimetric inequalities  for both compact and noncompact rank-1 symmetric spaces in this paper.  Our first result is  a sharp isoperimetric inequality for Neumann eigenvalues on bounded domains in compact rank-1 symmetric spaces.
      \begin{theorem} \label{thm1}
        Let 
        $\Omega$ be a bounded domain with smooth boundary,  contained in a geodesic ball of radius $\frac{i(M)}{4}$ in a compact rank-1 symmetric space $M$ with real dimension $m=kn$ (cf. Section \ref{sect2}). Let 
        \begin{align}\label{l}
            l=\begin{cases}
                m-1, &\text{if $k=1, m$}, \\
                \left[ \frac{m+k}{2} \right]-1 , & \text{if $1<k< m$},
            \end{cases}
        \end{align}
        where $\left[ \frac{m+k}{2} \right]$ is the  largest integer that does not exceed $\frac{m+k}{2}$. Then
        \begin{align}\label{1.5}
            \frac{1}{\mu_1(\Omega) } +\frac{1}{\mu_2(\Omega) } +\cdots+\frac{1}{\mu_l(\Omega) }  \geq \frac{l}{\mu_1(B)},
        \end{align}
        where $B$ is a geodesic ball in $M$ having  volume  $|\Omega|$.
        Furthermore, the equality of \eqref{1.5} occurs if and only if $\Omega$ is a geodesic ball.    
    \end{theorem}
    Our second result is a sharp isoperimetric inequality for Neumann eigenvalues in  noncompact rank-1 symmetric spaces. 
    \begin{theorem} \label{thm2}
        Let $\Omega$ be a bounded domain with smooth boundary in a noncompact rank-1 symmetric space $M$  with real dimension $m=kn$ (cf. Section \ref{sect2}). Then
        \begin{align}\label{1.6}
               \frac{1}{\mu_1(\Omega) } +\frac{1}{\mu_2(\Omega) } +\cdots+\frac{1}{\mu_{p}(\Omega) }  \geq \frac{p}{\mu_1(B)},
        \end{align}
    where $B$ is a geodesic ball in $M$ having  volume  $|\Omega|$, and
    \begin{align}\label{p}
        p=\begin{cases}
            k(n-1), & \text{if  $n>1$},\\
            kn-1, & \text{if $n=1$}.
        \end{cases}
    \end{align}
        Furthermore, the equality of \eqref{1.6} occurs if and only if $\Omega$ is a geodesic ball.    
    \end{theorem}
   Obviously, both Theorem \ref{thm1} and Theorem \ref{thm2} imply 
    $$
    \mu_1(\Omega)\le \mu_1(B),
    $$
    recovering the main results of \cite{AS96} proved by  Aithal and Santhanam. Moreover if $k=1$, i.e. $M$ is a hyperbolic space of dimension $n$, $p=n-1$, hence Theorem  \ref{thm2} recovers Theorem 1.2 of \cite{WX23} proved by Wang and Xia.
    On bounded domains in rank-1 symmetric spaces, it  remains an interesting question that whether Ashbaugh and Benguria's conjecture (Conjecture \ref{con1}) holds  or not.
    
    The rest of the paper is organized as follows. In Section 2, we recall the geometry of rank-1 symmetric spaces. In Section 3, we give some useful tools needed in the proofs of main theorems. Sections 4 and 5 are devoted to proving  Theorem \ref{thm1} and Theorem \ref{thm2}.
    
\section{Geometry of Rank-1 Symmetric spaces}\label{sect2}
Let $M$ denote any one of the following rank-1 symmetric spaces with real dimension $m$: Complex projective space $\mathbb{C}\mathbb{P}^n$, quarternionic projective space $\mathbb{H} \mathbb{P}^n$, the Cayley projective plane $\mathbf{Ca} \mathbb{P}^2$ or their non-compact duals. Let $\mathbb{K}$ denote $\mathbb{R}$, $\mathbb{C}$,$\mathbb{H}$ or $\mathbf{Ca}$, and $k=\dim_{\mathbb{R}} \mathbb{K}$, then $m=kn$.

We recall some basic facts about geodesic polar coordinates $(r, \xi)$ in both compact and noncompact rank-1 symmetric spaces, see   \cite[Section 3]{AS96} and \cite[Section 2]{LWW22} for more details. Let $o\in M$  be the center of  geodesic polar coordinates, and $J(r)$ denote the Riemannian density function. Then 
\begin{align*}
        J(r)=\sin^{m-1} r \cos^{k-1}r
    \end{align*}
for $r\in (0, \pi/2)$ when  $M$ is compact;
    \begin{align*}
        J(r)=\sinh^{m-1} r \cosh^{k-1} r
    \end{align*}
 for $r>0$ when $M$ is noncompact.
The trace of the second fundamental form  of $S(r)$ is 
\begin{align*}
    H(r)=\frac{J'(r)}{J(r)},
\end{align*}
namely
\begin{align*}
    H(r)=\begin{cases}
    (m-1)\cot r- (k-1) \tan r, &\quad \text{if $M$ is compact},\\
       (m-1) \coth r+(k-1) \tanh r, &\quad \text{if $M$ is noncompact}.\\
    \end{cases}
\end{align*}
   The Laplace operator of $M$ is given by
    \begin{align*}
        \Delta_M=\frac{\partial^2 }{\partial r^2 }+ H(r) \frac{\partial }{\partial r }+\Delta_{S_r}, 
    \end{align*}
    where $\Delta_{S_r}$ denotes the Laplacian of $S_r$. Moreover the first non-zero eigenvalue of $\Delta_{S_r}$ is
    \begin{align*}
        \lambda_1(S_r)=-H'(r),
    \end{align*}
    and the associated eigenfunctions are the linear coordinate functions restricted to $\mathbb{S}^{m-1}$, denoted by $\omega_i(\xi ) \, (1\leq i \leq m) $, satisfying
    \begin{align}\label{2.1}
        \sum_{i=1}^m |\nabla^{S_r} \omega_i(\xi) |^2 = -H'(r).
    \end{align}
Moreover it is known from the proof of \cite[Lemma 4.11]{CR19} that
   \begin{align*}
      |\nabla^{S_r} \omega_i(\xi)|^2 =\begin{cases}
   \sum\limits_{j=k+1}^m \frac{\langle \xi_i,\eta_j \rangle^2 }{\sin^2 r } +  \sum\limits_{j=2}^k \frac{\langle \xi_i,\eta_j \rangle^2 }{\sin^2 r \cos^2 r }, &\quad \text{if $M$ is compact},\\
\sum\limits_{j=k+1}^m \frac{\langle \xi_i,\eta_j \rangle^2 }{\sinh^2 r } +  \sum\limits_{j=2}^k \frac{\langle \xi_i,\eta_j \rangle^2 }{\sinh^2 r \cosh^2 r}, &\quad \text{if $M$ is noncompact},
    \end{cases}
\end{align*}
where  $\{ \eta_k \}_{k=1}^m$ is an orthonormal basis of $T_o M$ with $ \eta_1=( \omega_1(\xi), \cdots,\omega_m(\xi) )$ and $\eta_{i+1}=J \eta_i$ for $i=1,\cdots, k-1$. Then we estimate that for compact case 
    \begin{align}\label{2.2}
    \begin{split}
        |\nabla^{S_r} \omega_i(\xi)|^2 &= \sum_{j=k+1}^m \frac{\langle \xi_i,\eta_j \rangle^2 }{\sin^2 r } + \sum_{j=2}^k \frac{\langle \xi_i,\eta_j \rangle^2 }{\sin^2 r \cos^2 r } \\
       &\leq \begin{cases}
       \frac{1}{\sin^2 r \cos^2 r } ,& k\neq 1\\
       \frac{1}{\sin^2 r },& k=1
       \end{cases}
       \\&\leq  -\frac{1}{l}H'(r)
       \end{split}
    \end{align}
for $0<r\leq \frac{\pi}{4}$, 
and for noncompact case
	\begin{align}\label{2.3}
 \begin{split}
		|\nabla^{S_r} \omega_i(\xi)|^2 &= \sum_{j=k+1}^m \frac{\langle \xi_i,\eta_j \rangle^2 }{\sinh^2 r } + \sum_{j=2}^k \frac{\langle \xi_i,\eta_j \rangle^2 }{\sinh^2 r \cosh^2 r } \\
        & \leq \begin{cases}
       \frac{1}{\sinh^2 r }, &  k< m\\
       \frac{1}{\sinh^2 r \cosh^2 r },&  k=m
       \end{cases}\\
        & \leq -\frac{1}{p}H'(r),
        \end{split}
	\end{align}
    where $l$ and $p$ are defined by \eqref{l} and \eqref{p} respectively.
\section{Some mathematical tools needed}
    In this section,  we will give some properties of the first nonzero  eigenvalue  of Neumann Laplacian and corresponding eigenfunctions on round balls, and construct trial functions for lower order Neumann eigenvalues on bounded  domains  in rank-1 symmetric spaces.
    \subsection{Properties of eigenfunctions for geodesic balls} Let $M$ be a rank-1 symmetric space and $B$ be a round geodesic ball  in $M$ with radius $R$. If $M$ is compact, we assume further that $R< \frac \pi 4$. It is known from \cite[Section 3]{AS96}  that 
    \begin{align*}
        \mu_1(B)=\mu_2(B)=\cdots=\mu_m(B),
    \end{align*}
    and  the corresponding  eigenfunctions are given by
    \begin{align*}
        h_i(r,\xi)=g(r)\omega_i(\xi),\, i=1,2,\cdots,m,
    \end{align*}
    where $g(r)$ satisfies the following one dimensional boundary value problem
    \begin{equation} \label{3.1}
        \begin{cases}
            g''(r)+H(r)g'(r)+\big(\mu_1(B)+H'(r)\big)g(r)=0,\quad r\in (0, R), \\
            g(0)=0,\,g'(R)=0,
        \end{cases}
    \end{equation}
    and $\omega_i$'s are the restrictions of the linear coordinate functions on $\mathbb{S}^{m-1}$. Moreover we can choose a solution $g$ of \eqref{3.1} satisfying $g'(0)=1$, hence $g(r)>0$ in $(0, R]$ and $g'(r)>0$ in $[0, R)$. 
Furthermore it is easy to check that $\mu_1(B)$ is the first eigenvalue corresponding to the quotient: 
    \begin{align}\label{3-2}
       \mathcal{Q}(\vp)= \frac{\int_0^R \big(\vp'(r)^2-H'(r)\vp(r)^2\big) J(r)\, dr}{\int_0^R \vp^2(r) J(r)\, dr}
    \end{align}
with constraints $\vp(0)=0$ and $\vp\in C^1([0,R])$, where $J(r)$ is the Riemannian density function defined in Section \ref{sect2}.

    \subsection{Trial functions for lower order Neumann eigenvalues} From here on, we assume $\Omega$ to be a bounded domain with smooth boundary in a rank-1 symmetric space $M$. If $M$ is compact, we assume further that $\Omega$ is contained in a geodesic ball of radius $i(M)/4$.  Denote by $u_i$ an eigenfunction corresponding to $\mu_i(\Omega)$, then the Neumann eigenvalue $\mu_i(\Omega)$ can be characterized variationally by
    \begin{align}\label{3.2}
        \mu_i(\Omega)=\min\Big\{\frac{\int_\Omega |\nabla u|^2\, dx}{\int_\Omega u^2\, dx}: u\in H^1(\Omega)\setminus \{0\}, u\in \operatorname{span}\{u_0, u_1, \cdots, u_{i-1}\}^\perp\Big\}.
    \end{align}
 Let $B\subset M$ be a geodesic ball such that $|\Omega|=|B|$, and denote by $R$ the radius of $B$. Define  $G(r):[0, \infty) \to [0, \infty)$ by 
        \begin{align}\label{3.3}
        G(r)=\begin{cases}
            g(r),&\quad r<R,\\
            g(R),&\quad r\geq R,
        \end{cases}
    \end{align}
    where $g(r)$ is defined by \eqref{3.1}.  By Brouwer's fixed point theorem, we can choose a proper origin $o\in \Omega$ such that
    \begin{align}\label{3.4}
        \int_{\Omega} G(r_o)\omega_i(\xi)\, dx=0,\, i=1,2,\cdots,m,
    \end{align}
    where $r_o(x)=\text{dist} (o,x)$ is the distance from $o$ to $x$ in $M$. 
    
    From now on, we fix the point $o$ so that \eqref{3.4} holds. Let $(r, \xi)$ denote the polar coordinates centered at $o$.
     For $1\le i\le m$ and  $1\le j \le m$, we define
    \begin{align}\label{3.5}
        v_i(x)=G(r)\omega_i(\xi),
    \end{align}
     and let
    \begin{align*}
    q_{ij}= \int_{\Omega} v_i(x) u_j(x) \, dx.\end{align*}
By  QR-factorization, there exists  an orthogonal matrix $U=(a_{ij})$ such that 
    \begin{align*}
        0=\int_{\Omega} \sum_{k=1}^n a_{ik} v_k(x)  u_j(x)\, dx
    \end{align*}
    for all $1\le j< i\le m$.
   Therefore we can choose a proper basis of $T_o M$  so that
    \begin{align}\label{3.6}
        \int_{\Omega} v_i(x)u_j(x)\, dx=0,
    \end{align}
    for $1\le j<i\le m$. According to \eqref{3.4} and \eqref{3.6}, we find that $v_i(x)$ is a trial function for $\mu_i(\Omega)$. Thus in view of \eqref{3.2}, we conclude that
    \begin{align}\label{3.7}
        \int_{\Omega} v_i(x)^2\, dx \leq \frac{1}{\mu_i(\Omega) } \int_{\Omega} |\nabla v_i(x) |^2\, dx
    \end{align}
    for $1\le i\le m$.
\section{proof of theorem \ref{thm1}}
In this section,  we consider the case of compact rank-1 symmetric spaces  and prove Theorem \ref{thm1}. To begin with, we recall the  following lower bound for the first nonzero Neumann eigenvalue on geodesic balls \cite[Corollary 1]{AS96}, which will be used later on.
    \begin{lemma} [\cite{AS96}] \label{lm4.1}
       Let $M$ be a rank-1 symmetric space of compact type (see Section \ref{sect2}), and $B\subset M$ be a geodesic ball of radius $R$. If $R\le \frac \pi 4$, then
       \begin{align}\label{4.1}
       \mu_1(B)\geq 2(m+k).
       \end{align}
    \end{lemma}
 \begin{lemma}\label{lm4.2}
 Let $G(r)$ be the function defined in \eqref{3.3}. Then the function
 $$
 \varphi(r):=\frac{G(r)}{\sin r \cos r}
 $$
is nonincreasing in $(0,\frac{\pi}{4})$.
    \end{lemma}
    \begin{proof}
   If $r>R$, then $G(r)=G(R)$, hence $\vp(r)$ is strictly decreasing in $(R, \frac \pi 4)$. Now we focus on the case of $r\in [0,R]$. Direct calculation gives
   \begin{align*}
       \vp'(r)=\frac{2}{\sin r\cos r}\Big(g'(r)-2\cot 2r g(r)\Big),
   \end{align*}
and it suffices to show
        \begin{align}\label{4.2}
           s(r):=g'(r) -2\cot 2r g(r)\leq 0
        \end{align}
        for $r\in (0, R)$.
        Note $s(0^+)=0$ and $s(R)<0$, and then
        we assume by the contradiction that $s(r)$ has a maximum point  $r_0 \in (0, R)$ and $s(r_0)>0$. So $s'(r_0)=0$, namely
        \begin{align*}
            g''(r_0)-(\cot r_0-\tan r_0) g'(r_0) +\frac{1}{\sin^2 r_0 \cos ^2 r_0} g(r_0)=0.
        \end{align*}
   Applying the equation \eqref{3.1}, the above equality is equivalent to 
        \begin{align}\label{4.3}
            (m\cot r_0-k\tan r_0 )g'(r_0) +\Big(\mu_1(B)-\frac{m-k}{\sin^2 r_0 }-\frac{k}{\sin^2 r_0 \cos ^2 r_0 }\Big )g(r_0)=0.
        \end{align}
      The assumption $s(r_0)>0$ gives
        \begin{align*}
            g'(r_0)>2\cot 2 r_0 g(r_0),
        \end{align*}
and combining with equality \eqref{4.3}, we obtain that
        \begin{align}\label{4.4}
            \Big(-\mu_1(B) +  \frac{m-k}{\sin^2 r_0 }+\frac{k}{\sin^2 r_0 \cos ^2 r_0 }\Big) -(\cot r_0- \tan r_0)(m\cot r_0-k \tan r_0)>0.
        \end{align}
        A brief calculation shows that  \eqref{4.4} is equivalent to 
        \begin{align*}
          -\mu_1(B) + 2(m+k)>0,
        \end{align*}
        contradicting with Lemma \ref{4.1}. Hence \eqref{4.2} comes true, proving the lemma.
    \end{proof}
Now we turn to prove Theorem \ref{thm1}.
\begin{proof}[Proof of Theorem \ref{thm1}]
    Let $v_i(x)$ ($1\le i \le m$) be trial functions defined by \eqref{3.5}, and  
    \begin{equation} \label{4.5}
        |\nabla^M v_i(x)|^2 = |G'(r)|^2 \omega_i(\xi)^2 + G(r)^2 |\nabla^{S_r} \omega_i(\xi)|^2,
    \end{equation}
    where $\nabla^M$ denotes the gradient operator of $M$, and $\nabla^{S_r}$ denotes the gradient operator of $S_r$ with the induced metric. 
    Plugging \eqref{4.5} into \eqref{3.7} and summing over $i$, we have
    \begin{align}\label{4.6}
\int_{\Omega} |G(r)|^2\, dx \leq & \sum_{i=1}^m \frac{1}{\mu_i(\Omega)} \int_{\Omega} |G'(r)|^2 |\omega_i(\xi)|^2 \, dx
       +\sum_{i=1}^m \frac{1}{\mu_i(\Omega)} \int_{\Omega} |G(r)|^2 |\nabla^{S_r} \omega_i(\xi)|^2 \, dx.
    \end{align}
  Recall that $G(r)$ is a constant for $r>R$, we estimate that
    \begin{align} \label{4.7}
    \begin{split}
        \int_{\Omega} |G'(r)|^2 |\omega_i(\xi)|^2 \, dx &= \int_{\Omega \cap B} |G'(r)|^2 |\omega_i(\xi)|^2 \, dx \\
        & \leq \int_{B} |G'(r)|^2 |\omega_i(\xi)|^2 \, dx \\
        & = \frac{1}{m} \int_{B} |G'(r)|^2 \, dx,
    \end{split}
    \end{align}
       where $B$ is  the geodesic ball of radius $R$ centered at $o$ in $M$ with $|B|=|\Omega|$. Moreover by the assumption that $\Omega$ is contained in a geodesic ball of radius $\frac{i(M)}{4}$, then we have $R\le \frac{\pi}{4}$.
 Using equality \eqref{2.1} and inequality \eqref{2.2} we estimate that
      \begin{align} \label{4.8}
      \begin{split}
        &\sum_{i=1}^m \frac{1}{\mu_i(\Omega) } |\nabla^{S_r} \omega_i(\xi)|^2 \\
        \le & \sum_{i=1}^{l} \frac{1}{\mu_i(\Omega) } |\nabla^{S_r} \omega_i(\xi)|^2+\frac{1}{\mu_{l+1}(\Omega) } \sum_{i=l+1}^{m} |\nabla^{S_r} \omega_i(\xi)|^2\\
         \leq  & \sum_{i=1}^{l} \frac{1}{\mu_i(\Omega) } |\nabla^{S_r} \omega_i(\xi)|^2+\frac{1}{\mu_{l+1}(\Omega) }\big(-H'(r)-\sum_{i=1}^{l}|\nabla^{S_r} \omega_i(\xi)|^2\big)  \\
         = & \sum_{i=1}^{l} \frac{1}{\mu_i(\Omega) } |\nabla^{S_r} \omega_i(\xi)|^2+\frac{1}{\mu_{l+1}(\Omega) }\sum_{i=1}^{l}\big(-\frac{H'(r)}{l}-|\nabla^{S_r} \omega_i(\xi)|^2\big)  \\
         \leq & \sum_{i=1}^{l} \frac{1}{\mu_i(\Omega) } |\nabla^{S_r} \omega_i(\xi)|^2+\sum_{i=1}^{l}\frac{1}{\mu_{i}(\Omega) }\big(-\frac{H'(r)}{l}-|\nabla^{S_r} \omega_i(\xi)|^2\big)  \\
        = & \frac{1}{l} \sum_{i=1}^{l} \frac{-H'(r)}{\mu_i(\Omega) },
        \end{split}
    \end{align}
    where $l$ is defined by \eqref{l}.
   Substituting inequalities \eqref{4.7} and \eqref{4.8} to  inequality \eqref{4.6} yields
    \begin{align}\label{4.9}
\int_{\Omega} G(r)^2 \,dx \leq  \sum_{i=1}^m \frac{1}{m \mu_i(\Omega) } \int_{B} G'(r)^2 \, dx  + \frac{1}{l} \sum_{i=1}^{l} \frac{1}{\mu_i(\Omega) }\int_{\Omega} G(r)^2 (-H'(r))\, dx.
    \end{align}
Recall from Lemma \ref{lm4.2} that $\frac{G(r)}{\sin r \cos r}$ is monotone nonincreasing in $(0, \frac \pi 4)$, then 
    \begin{align*}
        -G(r)^2H'(r)=(m-k)\frac{G(r)^2}{\sin^2 r} + (k-1)\frac{G(r)^2}{\sin^2 r \cos^2 r} 
    \end{align*}
    is nonincreasing in $(0, \frac \pi 4)$ either, hence we have
    \begin{align}\label{4.10}
        \int_{\Omega} |G(r)|^2 (-H'(r) )\, dx\le   \int_{B} |G(r)|^2 (-H'(r) )\, dx.
    \end{align}
    By the definition of $G(r)$, we see easily that $G(r)$ is monotone increasing in  $(0, \frac \pi 4)$, so
    \begin{align}\label{4.11}
         \int_{B} |G(r)|^2 \, dx
        \leq & \int_{\Omega} |G(r)|^2\, dx.
    \end{align}
   Putting inequalities \eqref{4.9}, \eqref{4.10} and \eqref{4.11} together, we conclude that
    \begin{align*} 
    \begin{split}
        \int_{B} G(r)^2 \, dx
        \leq \frac{1}{l} \sum_{i=1}^{l} \frac{1}{\mu_i(\Omega) } \Big(\int_{B} |G'(r)|^2 -|G(r)|^2 H'(r) \, dx\Big),
        \end{split}
    \end{align*}
yielding 
    \begin{align*}
        \sum_{i=1}^{l} \frac{1}{\mu_i(\Omega) } 
        \geq  \frac{l\int_{B} |G(r)|^2 \, dx }{\int_{B} G'(r)^2 -G(r)^2 H'(r) \, dx} 
        = \frac{l}{\mu_1(B)},
    \end{align*}
    where in the equality we used the fact that $g(r)$ ($r\in [0, R]$) is the first eigenfunction with respect to $\mu_1(B)$, as characterized by \eqref{3-2}.  Hence we complete the proof of  estimate \eqref{1.5}.

   Moreover if the equality of \eqref{1.5} occurs, all inequalities above hold as equalities. Particularly \eqref{4.11} holds as an equality, then $\Omega=B$.
\end{proof}

\section{proof of theorem \ref{thm2}}
In this section, we assume $M$ is a noncompact rank-1 symmetric space. Let $\Omega$ be a bounded domain with smooth boundary in $M$, and $B$ be a geodesic ball in $M$ with volume $|\Omega|$. Denote by $R$ be the radius of $B$. The proof of Theorem \ref{thm2} is quite similar as that of Theorem \ref{thm1}. To begin with, we first give a monotonicity lemma.
\begin{lemma}\label{lm5.1}
 Let $G(r)$ be the function defined in \eqref{3.3}. Then the function
 $$
 \varphi(r):=\frac{G(r)}{\sinh r}
 $$
is nonincreasing in $(0, \infty)$.
    \end{lemma}
\begin{proof}
Note that $G(r)=G(R)$ for $r>R$, then $\vp(r)$ is monotone decreasing in $(R,\infty)$. Now we focus on the case  of $(0,R)$. Taking derivative of $\vp(r)$ yields
\begin{align*}
    \vp'(r)=\frac{1}{\sinh^2 r}\big(G'(r)\sinh r-G(r)\cosh r\big),
\end{align*}
and it suffices to show that for $r\in(0, R)$ 
        \begin{align}\label{5.1}
            s(r):= G'(r) - \coth r G(r) \leq 0.
        \end{align}
Noting that $s(0^+)=0$ and $s(R)<0$, we then assume by  contradiction that   $s(r)$ has a maximum point  $r_0 \in (0, R)$ and $s(r_0)>0$. Then  $s'(r_0)=0$, i.e.
        	\begin{align}\label{5.2} 
		g''(r_0)-\coth r_0 g'(r_0)+\frac{g(r_0)}{\sinh^2 r_0}=0.
	\end{align}
 Combining equality \eqref{5.2} and equation \eqref{3.1}, we have
 	\begin{align}\label{5.3}
 		(m \coth r_0 +(k-1) \tanh r_0 )g'(r_0)
   +(\mu_1(B)-\frac{m-k+1 }{\sinh^2 r_0 }-\frac{k-1 }{\sinh^2 r_0 \cosh^2 r_0 } )g(r_0)=0.
 	\end{align}
 Since $s(r_0)>0$, then
 	\begin{align*}
 		g'(r_0)-\coth r_0 g(r_0)>0.
 	\end{align*}
Thus it follows from above inequality and equality \eqref{5.3} that
 	\begin{align*}
 		-\mu_1(B) + \frac{m-k+1}{\sinh^2 r_0}+ \frac{k-1}{\sinh^2 r_0 \cosh^2 r_0}> m \coth^2 r_0 +k-1,
 	\end{align*}
 	that is
 	\begin{align}\label{5.4}
 	\frac{m}{\sinh^2 r_0}(1-\cosh ^2 r_0) +(k-1)\big(\frac{1-\cosh^2 r_0}{\sinh^2 r_0 \cosh^2 r_0} -1\big)>\mu_1(B).
 	\end{align}
  Clearly the left hand side of \eqref{5.4} is negative, 
 	and thus \eqref{5.4} is contradicting with the fact $\mu_1(B)>0$. Hence inequality \eqref{5.1} holds.	   
\end{proof}
Now we turn to prove Theorem \ref{thm2}.
\begin{proof}[Proof of Theorem \ref{thm2}]
     Let $v_i(x)$ ($1\le i \le m$) be trial functions defined by \eqref{3.5}. Then similarly as in the proof of Theorem \ref{thm1}, we have
         \begin{align}\label{5.5}
 \int_{\Omega} G(r)^2 \, dx \leq & \sum_{i=1}^m \frac{1}{\mu_i(\Omega)} \int_{\Omega} G'(r)^2 \omega_i(\xi)^2\, dx
         +\sum_{i=1}^m \frac{1}{\mu_i(\Omega)} \int_{\Omega} |G(r)|^2 |\nabla^{S_r} \omega_i(\xi)|^2\, dx,
    \end{align}
and 
    \begin{equation} \label{5.6}
        \int_{\Omega} |G'(r)|^2 |\omega_i (\xi) |^2 \, dx \leq \frac{1}{m} \int_{B} G'(r)^2 \, dx,
    \end{equation}
      where $B$ is  the geodesic ball of radius $R$ centered at $o$ in $M$ with $|B|=|\Omega|$.
Using \eqref{2.1} and \eqref{2.3} we estimate that
    \begin{align} \label{5.7}
    \begin{split}
        &\sum_{i=1}^m \frac{1}{\mu_i(\Omega) } |\nabla^{S_r} \omega_i(\xi)|^2  \\
        \le & \sum_{i=1}^{p} \frac{1}{\mu_i(\Omega) } |\nabla^{S_r} \omega_i(\xi)|^2+\frac{1}{\mu_{p+1}(\Omega) } \sum_{i=p+1}^{m} |\nabla^{S_r} \omega_i(\xi)|^2,  \\
         \leq & \sum_{i=1}^{p} \frac{1}{\mu_i(\Omega) } |\nabla^{S_r} \omega_i(\xi)|^2+\frac{1}{\mu_{p+1}(\Omega) }\big(-H'(r)-\sum_{i=1}^{p}|\nabla^{S_r} \omega_i(\xi)|^2\big)  \\
         = & \sum_{i=1}^{p} \frac{1}{\mu_i(\Omega) } |\nabla^{S_r} \omega_i(\xi)|^2+\frac{1}{\mu_{p+1}(\Omega) }\sum_{i=1}^{p}\big(-\frac{H'(r)}{p}-|\nabla^{S_r} \omega_i(\xi)|^2\big) \\
        \le & \sum_{i=1}^{p} \frac{1}{\mu_i(\Omega) } |\nabla^{S_r} \omega_i(\xi)|^2+\sum_{i=1}^{p}\frac{1}{\mu_{i}(\Omega) }\big(-\frac{H'(r)}{p}-|\nabla^{S_r} \omega_i(\xi)|^2\big) \\
       =& \frac{1}{p} \sum_{i=1}^{p} \frac{-H'(r)}{\mu_i(\Omega) },
        \end{split}
    \end{align}
   where $p$ is defined by \eqref{p}.
   Assembling inequalities \eqref{5.5}, \eqref{5.6} and \eqref{5.7}, we get
    \begin{align}\label{5.8}
 \int_{\Omega} G(r)^2 \, dx \leq&\sum_{i=1}^m \frac{1}{m \mu_i(\Omega) } \int_{B} G'(r)^2 \, dx+ \frac{1}{p} \sum_{i=1}^{p} \frac{1}{\mu_i(\Omega) }\int_{\Omega} G(r)^2 (-H'(r))\, dx.
 \end{align}
 Recall from Lemma \ref{lm5.1} that $G(r)/\sinh r$ is nonincreasing in $(0, \infty)$, so is
     $$
     G(r)^2(-H'(r))=(m-k)\frac{G(r)^2}{\sinh^2 r} + (k-1)\frac{G(r)^2}{\sinh^2 r \cosh^2 r}.
     $$
Hence
 	\begin{align}\label{5.9}
 		\int_{\Omega} -G(r)^2 H'(r)\, dx \leq \int_{B} -G'(r)^2 H'(r)\, dx.
 	\end{align}
  Since $G(r)$ is monotone increasing in $(0, R)$, then
  \begin{align}\label{5.10}
  	\int_{B} G(r)^2 \, dx &\leq \int_{\Omega} G(r)^2 \, dx.
  \end{align}
 So it follows from \eqref{5.8}, \eqref{5.9} and \eqref{5.10} that 
 	\begin{align*}
 			\int_{B} G(r)^2\, dx \leq \frac{1}{p} \sum_{i=1}^{p}\frac{1}{\mu_i(\Omega) } \int_{B} G'(r)^2 -G(r)^2 H'(r)\, dx,
 	\end{align*}
 implies
 	\begin{align*}
 		\sum_{i=1}^{p} \frac{1}{\mu_i(\Omega)} \geq \frac{	p \int_{B} G(r)^2\, dx }{\int_{B} G'(r)^2 -G(r)^2 H'(r)\, dx } = \frac{p}{\mu_1(B)},
 	\end{align*}
  proving \eqref{1.6}. If the equality case of \eqref{1.6} occurs, then inequality \eqref{5.10} holds as an equality, hence $\Omega=B$.
\end{proof}

\section*{Data Availability}
Data sharing is not applicable to this article as no new data were created or analyzed in this study.
	\bibliographystyle{plain}
	\bibliography{ref}

\end{document}